\documentclass[12pt]{amsart}

\title{Quantum homology of compact convex symplectic manifolds}
\author{Sergei Lanzat}
\address{Max-Planck-Institut f\"{u}r Mathematik, 53111 Bonn, Germany}
\email{serjl@mpim-bonn.mpg.de}
\date{\today}
\makeatletter

\usepackage{amsmath,amssymb,amsthm,amsfonts,mathrsfs,amscd,amsopn}
\usepackage{dsfont}
\usepackage{latexsym}
\usepackage[pdftex]{graphicx}
\input xy
\xyoption{all}
\usepackage{bbm}

\renewcommand{\(}{\left(}
\renewcommand{\)}{\right)}

\newcommand{\FAT}[1]{\mbox{{$\mathbb{#1}$}}}
\newcommand{\Fat}[1]{\mbox{{$\scriptstyle\mathbb{#1}$}}}
\newcommand{\CL}[1]{\mbox{{$\mathcal{#1}$}}}

\newcommand{\KL}[1]{\mbox{{$\mathscr{#1}$}}}
\newcommand{\Kl}[1]{\mbox{{$\scriptstyle\mathscr{#1}$}}}
\newcommand{\til}[1]{\widetilde{#1}}
\renewcommand{\hat}[1]{\widehat{#1}}

\newcommand{\ZZ}{\FAT{Z}}
\newcommand{\NN}{\FAT{N}}
\newcommand{\FF}{\FAT{F}}
\newcommand{\DD}{\FAT{D}}
\newcommand{\QQ}{\FAT{Q}}
\newcommand{\KK}{\FAT{K}}
\newcommand{\PP}{\FAT{P}}

\newcommand{\BB}{\FAT{B}}

\newcommand{\RR}{\FAT{R}}

\newcommand{\CC}{\FAT{C}}
\newcommand{\zz}{\Fat{Z}}

\newcommand{\ff}{\Fat{F}}
\newcommand{\nn}{\Fat{N}}
\newcommand{\qq}{\Fat{Q}}

\newcommand{\IFF}{\Leftrightarrow}
\newcommand{\then}{\Rightarrow}

\newcommand{\cl}[1]{\overline{#1}}
\newcommand{\minus}{\smallsetminus}

\newcommand{\ve}{\varepsilon}

\newcommand{\la}{\langle}
\newcommand{\ra}{\rangle}

\newcommand{\del}{\partial}

\newcommand{\Id}{\mathds{1}}

\newcommand{\be}{\begin{itemize}}
\newcommand{\ee}{\end{itemize}}

\newcommand{\beq}{\begin{equation}}
\newcommand{\eeq}{\end{equation}}

\newcommand{\beqn}{\begin{equation}\nonumber}

\newcommand{\bea}{\begin{equation}\begin{aligned}}
\newcommand{\eea}{\end{aligned}\end{equation}}

\newcommand{\bean}{\begin{equation}\nonumber\begin{aligned}}

\DeclareMathOperator{\ev}{ev}
\DeclareMathOperator{\codim}{codim}

\DeclareMathOperator{\Ham}{Ham}

\newtheorem{thm}{Theorem}[section]
\newtheorem{thm*}{Theorem}
\newtheorem{lem}[thm]{Lemma}
\newtheorem{lem*}[thm*]{Lemma}
\newtheorem{prop}[thm]{Proposition}
\newtheorem{cor}[thm]{Corollary}

\newtheorem{defn}[thm]{Definition}
\newtheorem{thm-defn}[thm]{Theorem-Definition}

\theoremstyle{definition}
\newtheorem{rem}[thm]{Remark}
\newtheorem{exs}{Examples}
\newtheorem{ex}{Example}

\begin{document}

 \begin{abstract}
We study the space of pseudo-holomorphic spheres in compact symplectic manifolds with convex boundary. We show that the theory of Gromov-Witten invariants can be extended to the class of semi-positive manifolds with convex boundary. This leads to a deformation of intersection products on the absolute  and relative singular homologies. As a result,  absolute  and relative  quantum homology algebras are defined analogously to the case of closed symplectic manifolds. In addition, we prove the Poincar\'{e}-Lefschetz duality for the absolute  and relative  quantum homology algebras.

 \end{abstract}

\keywords{pseudo-holomorphic curves, Gromov-Witten invariants, quantum homology, convex symplectic manifolds.}
%
%

\maketitle

\section{Introduction and results.}
\subsection{History and motivation}
The celebrated work of Gromov \cite{Gr} opened the era of invariants of a symplectic manifold based on the counting of  pseudo-holomorphic curves in it. The first example of such invariants, now  called the Gromov invariant, was defined by Gromov itself in the same paper \cite{Gr}, which led to the proof of the famous  Gromov's non-squeezing theorem: The symplectic ball of radius $r$ can not be symplectically
embedded into the cylinder $B^2_R\times\CC^n$ for $R < r$.
More subtle invariants, now known under the name of Gromov-Witten (GW) invariants, appeared later in physics as correlation functions in Witten's topological sigma models \cite{W}, and were put on a solid mathematical basis by Ruan and Tian in  \cite{Ruan}, \cite{RT}. See also \cite{MS2},  \cite{MS3}. One of the consequences of GW invariants is the theory of quantum homology of a symplectic manifold, i.e. the deformation of the intersection product in its ordinary homology. Historically, these kinds of structures were considered in the case of closed (semi-positive) symplectic manifolds.

The purpose of this paper is to extend the theory of Gromov-Witten invariants and quantum homology to the class of semi-positive convex compact symplectic manifolds.

\subsection{Setting}
Consider a $2n$-dimensional compact symplectic manifold $(M,\omega)$ with non-empty boundary
$\del M$. Recall the following

\begin{defn}\label{Def:Convex symplectic manifolds}(cf. \cite{F-S}, \cite{McDuff}, \cite{MS3})
\mbox{}
 The boundary  $\del M$  is called convex if there exists a Liouville vector field $X$ (i.e. $\CL{L}_X\omega=d\iota_X\omega=\omega$), which is defined in the neighborhood of $\del M$ and which is everywhere transverse to $\del M$, pointing outwards; equivalently, there exists a $1$-form $\alpha$ on $\del M$ such that $d\alpha=\omega\mid_{\del M}$ and such that $\alpha\wedge(d\alpha)^{n-1}$ is a volume form inducing the boundary orientation of $\del M\subset M$. Therefore, $(\del M,\ker\alpha)$ is a contact manifold, and that is why a convex boundary is also called of a contact type.

A compact symplectic manifold $(M,\omega)$ with non-empty boundary $\del M$ is convex if $\del M$ is convex.

A non-compact symplectic manifold $(M,\omega)$ is \textit{convex} if there exists an increasing sequence of compact convex submanifolds $M_i\subset M$ exhausting $M$, that is, $$M_1\subset M_2\subset\ldots\subset M_i\subset\ldots\subset M\;\;\;\text{and}\;\;\;\bigcup\limits_{i}M_i=M.$$

\end{defn}

\begin{exs}
\item[(1)]The standard $(\RR^{2n},\omega_0)$ exhausted by balls.
\item[(2)]Cotangent $r$-ball bundles $(\DD_r T^*X,\omega_{can})$.
\item[(3)]Cotangent bundles $(T^*X,\omega_{can})$ over a closed base, exhausted by ball bundles, $T^*X=\bigcup\limits_{r>0}\DD_r T^*X.$
\item[(4)]Stein manifolds $(M,J,f)$, where $(M,J)$ is an open complex manifold and $f:M\to\RR$ is a smooth exhausting plurisubharmonic function without critical points off a compact subset of $M$. Here, "exhausting" means that $f$ is proper and bounded from below, and "plurisubharmonic" means that $\omega_f:=-d(df\circ J)$ is a symplectic form with Levi form $\omega_f(v,Jv)>0$ for all $0\neq v\in TM $. Then the gradient $X=\nabla f$ with respect to the K\"{a}hler form $\omega_f\circ(id\times J)$ satisfies $\CL{L}_X\omega_f=d\iota_X\omega_f=\omega_f$. Furthermore, if the critical points of $f$ are in, say, $\{f<1\}$, then the manifold $M$ is exhausted by the exact convex symplectic manifolds $M_i=\{f\leq i\}$.

\item[(5)] Symplectic blow-up of the above examples at finitely many interior points.

\end{exs}

Recall  that an \textit{almost complex structure}  on a smooth $2n$-dimensional manifold $M$ is a section $J$ of the bundle $\mathrm{End}\ TM$ such that $J^2(x)=-\mathds{1}_{T_xM}$ for every $x\in M$. An almost complex structure $J$ on $M$ is called compatible with $\omega$ ( or $\omega$-compatible) if $g_J:=\omega\circ(\mathds{1}\times J)$ defines a riemannian metric on $M$.
Denote by $\CL{J}(M,\omega)$ the space of all $\omega$-compatible almost complex structures on $(M,\omega)$.

Given $(M,\omega)$ and  $J\in\CL{J}(M,\omega)$, then $(TM,J)$ becomes a complex vector bundle and, as such, its first Chern class $c_1(TM,J,\omega)\in H^2(M;\ZZ)$ is defined. Note that since the space $\CL{J}(M,\omega)$  is non-empty and contractible, (see \cite[ Proposition $4.1, (i)$]{MS}), the class $c_1(TM,J,\omega)$  does not depend on $J\in\CL{J}(M,\omega)$, and we shall denote it just by $c_1(TM,\omega)$.

Let $H_2^S(M)$ be the image of the Hurewicz homomorphism $\pi_2(M)\to H_2(M,\mathbb{Z})$. The homomorphisms  $c_1: H_2^S(M)\to\mathbb{Z}$ and $\omega: H_2^S(M)\to\mathbb{R}$ are given by $c_1(A):=c_1(TM,\omega)(A)$ and $\omega(A)=[\omega](A)$ respectively.
The \textit{minimal Chern number}  of a symplectic manifold $(M,\omega)$ is the integer $N>0$, such that $\mathrm{Im}(c_1)=N\cdot\ZZ$. If $c_1(A)=0$ for every $A\in H_2^S(M)$, then we define the minimal Chern number to be $N=\infty$.

\begin{defn}\label{defn: semi-positive manifolds}(cf. \cite{H-S}, \cite{McDuff})
\mbox{}
A  symplectic $2n$-manifold $(M,\omega)$ is called semi-positive, if $\omega(A)\leq 0$ for  any $A\in H_2^S(M)$ with $3-n\leq c_1(A)<0$.
\end{defn}
\noindent The semi-positivity can be characterized by the following
\begin{lem}\label{lemma: characterization of semi-positive symplectic manifolds}(E.g. \cite[ Lemma $1.1$]{H-S}.)\\
A  symplectic $2n$-manifold $(M,\omega)$ is semi-positive if and only if one of the following conditions is satisfied.
\be
\item[$(i)$] $\omega=\alpha c_1 $, for some $\alpha\geq 0$.

\item[$(ii)$]$c_1(A)=0$ for every $A\in H_2^S(M)$.

\item[$(iii)$] The minimal Chern number $N$ satisfies $N\geq n-2$.
\ee
\end{lem}

\subsection{Main results and the structure of the paper}
Throughout the paper we consider a semi-positive convex compact symplectic manifold $(M,\omega)$. We shall always work over the base field $\mathbb{F}$, which is either $\mathbb{Z}_2$ or $\mathbb{Q}$.

In Section 2 we analyze the moduli spaces of pseudo-holomorphic spheres in $M$. We show that due to the convexity and the semi-positivity, such moduli spaces are well defined and generically they are smooth orientable manifolds.

In Section 3 we develop  the following notion of \textit{genus zero Gromov-Witten invariant  relative to the boundary}.  Let us fix a finite number of pairwise distinct marked points $\mathbf{z}:=(z_1,\ldots,z_m)\in\(\mathbb{S}^2\)^m$. For every integer  $p\in\{0,1,\ldots,m\}$ we define the genus zero Gromov-Witten invariant $GW_{A,p,m}$  relative to the boundary as a $m$-linear (over $\mathbb{F}$) map  $GW_{A,p,m}: H_*(M;\mathbb{F})^{\times p}\times H_*(M,\partial M;\mathbb{F})^{\times (m-p)}\to\mathbb{F}$,
which roughly speaking counts the following geometric configurations.\\
If $\mathbb{F}=\mathbb{Z}_2$, choose  smooth cycles $f_i: V_i\to M$  representing classes $a_i\in H_*(M;\mathbb{Z}_2)$ for every $i=1,\ldots, p$, and choose relative smooth cycles $f_j:(V_j,\partial V_j)\to(M,\partial M)$  representing classes $a_j\in H_*(M,\partial M;\mathbb{Z}_2)$ for every $j=p+1,\ldots, m$,  such that all the maps are in general position. Then  $GW_{A,p,m}(a_1,\ldots, a_m)$ counts  the parity of $J$-holomorphic spheres in class $A\in H_2^S(M)$, such that  $z_i$ is mapped to $f_i(V_i)$ for every  $i=1,\ldots, p$, and $z_j$ is mapped to $f_j(V_j\smallsetminus\partial V_j)$ for every $j=p+1,\ldots, m$.\\
Now, let $\mathbb{F}=\mathbb{Q}$. For  non-zero  $a_i\in H_*(M;\mathbb{Q})$, $i=1,\ldots, p$, there exist non-zero $r_i\in\mathbb{Q}$ and  smooth cycles $f_i: V_i\to M$ representing $r_ia_i$, and for non-zero $a_j\in H_*(M,\partial M;\mathbb{Q}),\ j=p+1,\ldots, m$,  there exist non-zero $r_j\in\mathbb{Q}$ and relative smooth cycles $f_j:(V_j,\partial V_j)\to(M,\partial M)$  representing  $r_ja_j$, such that all the maps are in general position. Then the invariant  $GW_{A,p,m}(r_1a_1,\ldots, r_ma_m)$ counts  the algebraic number of $J$-holomorphic spheres in class $A\in H_2^S(M)$, such that  $z_i$ is mapped to $f_i(V_i)$ for every  $i=1,\ldots, p$, and $z_j$ is mapped to $f_j(V_j\smallsetminus\partial V_j)$ for every $j=p+1,\ldots, m$. Define $\displaystyle GW_{A,p,m}(a_1,\ldots, a_m):=\frac{1}{r_1\cdots r_m}GW_{A,p,m}(r_1a_1,\ldots, r_ma_m)$.

In section 4 we use the Gromov-Witten invariants $GW_{A,p,m}$ to deform the classical intersection products $\bullet_i, i=1,2,3$; cf. Definition~\ref{defn: Intersection products}. These deformations give rise to the absolute (that is $H_*(M;\mathbb{F})\otimes_{\scriptstyle\mathbb{F}}\Lambda$) and the relative (that is $H_*(M,\partial M;\mathbb{F})\otimes_{\scriptstyle\mathbb{F}}\Lambda$) quantum homology algebras of $(M,\omega)$, where $\Lambda$ is the appropriate Novikov ring. We prove the Poincar\'{e}-Lefschetz duality for the absolute  and relative  quantum homology algebras. Finally, we compute the quantum homology algebras for some manifolds $(M,\omega)$.

\subsection*{Acknowledgement.}
I would like to thank Michael Entov, who introduced me to this subject, guided and helped me a lot. I am grateful to Michael Polyak for his valuable suggestions and comments in the course of my work on this paper. This work was carried out at Max-Planck-Institut f\"{u}r Mathematik, Bonn, and
I would like to acknowledge its excellent research atmosphere and hospitality.

\section{J-holomorphic curves in semi-positive convex compact symplectic manifolds.}\label{subsection: J-holomorphic spheres}
\subsection{J-holomorphic curves near a convex boundary.}\label{subsection: non-perturbed J-spheres}
Let  $(M,\omega)$ be a $2n$-dimensional compact semi-positive convex symplectic manifold and $J\in\CL{J}(M,\omega)$  is an  $\omega$-compatible almost complex structure. In addition, let $(\Sigma, j_{\Sigma})$ be a closed Riemann surface with its standard complex structure $j_{\Sigma}$.
\begin{defn}\label{Def: J-hol sphere}
A $J$-holomorphic curve is a smooth map $u:\Sigma\to M$ such that $$du\circ j_{\Sigma}=J\circ du.$$
\end{defn}

The following important lemmas show that $J$-holomorphic spheres for a suitable $J$ are bounded away from the boundary, i.e., they  lie  in the complement of some open neighborhood of $\del M$.

\begin{lem}\label{Lemma:convex boundary has Stein nieghborhood}(E.g. \cite[ Lemma $9.2.7$]{MS3}.)\\
Let  $(M,\omega)$ be a $2n$-dimensional compact convex symplectic manifold and let $X$ be a  Liouville vector field (see, Definition~\ref{Def:Convex symplectic manifolds}). Then there exists a pair $(f,J),$ where  $J\in\CL{J}(M,\omega)$  is an  $\omega$-compatible almost complex structure and $f:M\to(-\infty,0]$ is a proper smooth function such that
\be
\item For all $x\in\del M$ and for all $v\in T_x\del M$ we have $$J(x)X(x)\in T_x\del M,\;\;\;\omega(v,J(x)X(x))=0,\;\;\;\omega(X(x),J(x)X(x))=1.$$
\item $\del M=f^{-1}(\{0\})$.
\item There exists a collar neighborhood $W\cong(-\varepsilon,0]\times\del M$ of the boundary for some $\varepsilon>0$, which depends on the Liouville vector field $X$ such that  $\omega=-d(df\circ J)$ on $W$.
\ee
\end{lem}\label{Lemma: set of adopted is connected}

Such an almost complex structure $J$ is said to be \textit{adapted to the boundary} or \textit{\emph{contact}}, and such a function $f$ is called \textit{plurisubharmonic}.   The set of $\omega$-compatible  almost complex structures that are adapted to the boundary will be denoted by $\CL{J}(M,\del M,\omega)$. According to the lemma~\ref{Lemma:convex boundary has Stein nieghborhood} and \cite[ Remark $4.1.2$]{BPS}, or \cite[ discussion on the page $106$ ]{CFH} we have the following

\begin{lem}
Let  $(M,\omega)$ be a $2n$-dimensional compact convex symplectic manifold. Then the  set $\CL{J}(M,\del M,\omega)$ of $\omega$-compatible  almost complex structures that are adapted to the boundary is non-empty and connected.
\end{lem}

The next lemma (see also \cite[ Lemma $9.2.9$]{MS3}) and its corollaries are crucial for  the study of $J$-holomorphic spheres. Let us denote by $\Delta:=\del^2_s+\del^2_t$ the standard Laplacian.

\begin{lem}\label{lemma:J-hol curves are subharmonic}
Let  $J\in\CL{J}(M,\omega)$, $\Omega\subseteq\CC$ be an open set, $u:\Omega\to M$ be a $J$-holomorphic curve, and $f:M\to\RR$ is a smooth function such that $\omega=-d(df\circ J)$ on a neighborhood of the image of $u$. Then
$$\Delta(f\circ u)=\omega(\del_su,J\del_su)$$
and hence $f\circ u$ is subharmonic, i.e.  $\Delta(f\circ u)\geq 0$.
\end{lem}

Let  $(M,\omega)$ be a  compact convex symplectic manifold. We have the following

\begin{cor}\label{corollary:J-curvrs cannot touch the boundary}(See \cite[ Corollary $9.2.10$]{MS3}.)\\
Let $(\Sigma,j)$ be a connected closed Riemann surface, $W\subset M$ be an open neighborhood of $\del M$, and $u:\Sigma\to M$ be a smooth map, whose restriction to $u^{-1}(W)$ is $J$-holomorphic for some $J\in\CL{J}(M,\del M,\omega)$. Then $u(\Sigma)\cap\del M\neq\varnothing$ if and only if $u(\Sigma)\subset\del M$.
\end{cor}

\begin{cor}\label{corollary:J-spheres are bounded away from the boundary}
Let  $W$ be an open neighborhood of $\del M$,  such that  $\omega=-d(df\circ J)$ on $W$, where the pair $(f,J),$ satisfies the conditions of lemma~\ref{Lemma:convex boundary has Stein nieghborhood}. Then for any non-constant $J$-holomorphic sphere  $u:\mathbb{S}^2\to M$ we have  $u(\mathbb{S}^2)\subset M\minus W$.
\end{cor}

On can also consider the solutions of the perturbed Cauchy-Riemann equation $$\cl{\del}_J(u)=\nu(u),$$ where an inhomogeneous term $\nu(u)\in\Omega_J^{0,1}(\Sigma,u^*TM)$ is Hamiltonian with compact support. Recall, if  $(M,\omega)$ is a symplectic manifold and $H\in\CL{C}^{\infty}(M)$ is a smooth function on $M$, called a \textit{Hamiltonian}, one defines the smooth  vector field $X_H\in\Gamma(M,TM)$ on $M$ by the formula $$\omega(X_H,\cdot)=-dH(\cdot).$$ Such a vector field is called a \textit{Hamiltonian vector field}, and the space of Hamiltonian vector fields on $M$ will be denoted by $\Gamma_{\Ham}(M,TM)$.

Let  $(M,\omega)$ be a  compact convex symplectic manifold. Denote by $\CL{C}^{\infty}_c(M)\subseteq\CL{C}^{\infty}(M)$ the space of compactly supported Hamiltonians, i.e., $$H\in\CL{C}^{\infty}_c(M)\IFF\mathrm{supp}(H)\Subset\overset{\circ}{M}:=M\minus\del M.$$  Suppose we have a $1$-form
$$\KL{H}\in\Omega^1\(\Sigma,\CL{C}^{\infty}_c(M)\):=\Gamma\(\Sigma,\mathrm{Hom}\(T\Sigma,\CL{C}^{\infty}_c(M)\)\)$$ on $\Sigma$ with values in the linear space $\CL{C}^{\infty}_c(M)$.  Define the support of $\KL{H}$ to be
$$\mathrm{supp}\(\KL{H}\):=\bigcup\limits_{z\in\Sigma,\ \xi\in T_z\Sigma}\mathrm{supp}\(\KL{H}(z)(\xi)\).$$  Denote by $$\Omega_c^1\(\Sigma,\CL{C}^{\infty}_c(M)\)\subseteq\Omega^1\(\Sigma,\CL{C}^{\infty}_c(M)\)$$ the subspace of compactly supported $1$-forms, i.e., $$\KL{H}\in\Omega_c^1\(\Sigma,\CL{C}^{\infty}_c(M)\)\ \IFF\ \mathrm{supp}\(\KL{H}\)\Subset\overset{\circ}{M}.$$
Such a $1$-form gives rise to the $1$-form $$\CL{X}_{\Kl{H}}\in\Omega^1\(\Sigma,\Gamma_{\Ham}(M,TM)\):=\Gamma\(\Sigma,\mathrm{Hom}\(T\Sigma,\Gamma_{\Ham}(M,TM)\)\)$$  on $\Sigma$ with values in the linear space of Hamiltonian vector fields. It is given by $$T_z\Sigma\ni\xi\mapsto\CL{X}_{\Kl{H}}(z)(\xi):=X_{\Kl{H}(z)(\xi)}\in\Gamma_{\Ham}(M,TM),$$ and its support satisfies
$$\mathrm{supp}\(\CL{X}_{\Kl{H}}\):=\bigcup\limits_{z\in\Sigma,\ \xi\in T_z\Sigma}\mathrm{supp}\(X_{\Kl{H}(z)(\xi)}\)\Subset\overset{\circ}{M}.$$
Given a smooth map $u:\Sigma\to M$, denote by $$\CL{X}_{\Kl{H}}(u)\in\Omega^1(\Sigma,u^*TM):=\Gamma\(\Sigma,\mathrm{Hom}\(T\Sigma,u^*TM\)\)$$ the $1$-form along $u$ with values in the pullback tangent bundle. It is given by $$T_z\Sigma\ni\xi\mapsto\CL{X}_{\Kl{H}}(u)(z)(\xi):=X_{\Kl{H}(z)(\xi)}(u(z))\in T_{u(z)}M.$$
Finally, fix  $J\in\CL{J}(M,\del M,\omega)$ and define the inhomogeneous term $\nu$ by $$\nu:=-\(\CL{X}_{\Kl{H}}\)^{0,1}\;\;\text{and}\;\;\nu(u):=-\(\CL{X}_{\Kl{H}}(u)\)^{0,1},$$ where $\(\CL{X}_{\Kl{H}}\)^{0,1}$ and  $\(\CL{X}_{\Kl{H}}(u)\)^{0,1}$ denote the $J$-complex anti-linear parts of
$\CL{X}_{\Kl{H}}$ and $\CL{X}_{\Kl{H}}(u)$ respectively.
Then the perturbed Cauchy-Riemann equation has the form
\beq\label{equation:perturbed C-R}
\cl{\del}_{J,\Kl{H}}(u):=\cl{\del}_J(u)+\(\CL{X}_{\Kl{H}}(u)\)^{0,1}=0
\eeq
A  smooth map $u:\Sigma\to M$ that satisfies Equation~\eqref{equation:perturbed C-R} will be called  $(J,\KL{H})$-holomorphic curve.
Note that the form $\cl{\del}_{J,\Kl{H}}(u)$ is the $J$-complex anti-linear part of the covariant derivative $$d_{\Kl{H}}(u):=du+\CL{X}_{\Kl{H}}(u).$$ Since $\cl{\del}_{J,\Kl{H}}(u)=\cl{\del}_{J}(u)$ on $M\minus\mathrm{supp}\(\KL{H}\)$ we have the following corollary.

\begin{cor}\label{cor:perturbed J-spheres are bounded away from the boundary}
Let  $W$ be an open neighborhood of $\del M$, such that  $\omega=-d(df\circ J)$ on $W$, where the pair $(f,J),$ satisfies the conditions of Lemma~\ref{Lemma:convex boundary has Stein nieghborhood}. Suppose that  $\KL{H}\in\Omega_c^1\(\mathbb{S}^2,\CL{C}^{\infty}_c(M)\)$ such that $W\subseteq M\minus\mathrm{supp}\(\KL{H}\)$.  Then for any non-constant $(J,\KL{H})$-holomorphic sphere  $u:\mathbb{S}^2\to M$ we have  $u(\mathbb{S}^2)\subset M\minus W$.
\end{cor}

Due to Gromov, one can view a perturbed  $(J,\KL{H})$-holomorphic curve \\$u:\Sigma\to M$  as a $\(j_{\Sigma},\widetilde{J}\)$-holomorphic map (section)  $$\widetilde{u}:\Sigma\to\widetilde{M}:=\Sigma\times M,\;\;\;\;\widetilde{u}(z)=(z,u(z)),$$ where the almost complex structure $\widetilde{J}$ on $\widetilde{M}$ is given by $$\widetilde{J}:=\begin{pmatrix} j_{\Sigma}&0\\&\\-\(\CL{X}_{\Kl{H}}\)^{0,1}\circ  j_{\Sigma}&J\end{pmatrix}.$$ We have that the projection $$\pi:\(\widetilde{M},\widetilde{J}\)\to\(\Sigma,j_{\Sigma}\)$$ is $\(\widetilde{J},j_{\Sigma}\)$-holomorphic and for each $z\in\Sigma$, the  fibre
$$(\{z\}\times M, J)\subseteq\(\widetilde{M},\widetilde{J}\)$$ is an almost complex submanifold. Moreover, there exits a symplectic form $\omega_{\Sigma}$ on $\Sigma$, such that $J$ is $\omega$-compatible if and only if $\widetilde{J}$ is $(\omega_{\Sigma}\oplus\omega)$-compatible. We have seen in the Corollary~\ref{cor:perturbed J-spheres are bounded away from the boundary} that for any compactly supported Hamiltonian perturbation $\KL{H}$,  a  $(J,\KL{H})$-holomorphic curve  lies in the complement of some open neighborhood of the boundary $\del M$. Hence, the respective  $\(j_{\Sigma},\widetilde{J}\)$-holomorphic section  lies  in the complement  of some open neighborhood of the boundary $\del \widetilde{M}=\Sigma\times\del M$.

It follows from corollaries \ref{corollary:J-curvrs cannot touch the boundary} and \ref{corollary:J-spheres are bounded away from the boundary} that the Gromov compactness theorem holds for both non-perturbed and perturbed $J$-holomorphic curves. See \cite[ Chapter $4$]{MS3}, \cite[ Section $6$]{PW}, \cite[ Section $3$]{Ruan}, \cite[ Theorem $4.1.1$]{AL}.

\subsection{Moduli spaces of J-holomorphic spheres.}
Let  $(M,\omega)$ be a $2n$-dimensional semi-positive compact convex symplectic manifold and let  $J\in\CL{J}(M,\del M,\omega)$. From the above discussion it follows that one can repeat verbatim the construction of moduli spaces of (perturbed) J-holomorphic spheres. The complete proves can be found in  \cite[ Chapter $3$]{MS3}. Let us just restate the central results.

Recall that a smooth map $u:\mathbb{S}^2=\CC P^1\to M$ is called \textit{simple} if  $u=w\circ\phi$ with meromorphic  $\phi:\CC P^1\to\CC P^1$ implies $\mathrm{deg}\ \phi =1$. Given a class $A\in H_2^S(M)$, denote by $\CL{M}(A,J)$ the set of all $J$-holomorphic spheres, which represent the class $A$. Denote also  by $\CL{M}_s(A,J)\subseteq\CL{M}(A,J)$ the subset of simple $J$-holomorphic spheres, which represent the class $A$.

\begin{thm}\label{thm:moduli space of simple curves is a manifold}
There exists a subset $\CL{J}_{reg}:=\CL{J}_{reg}(A)\subset\CL{J}(M,\del M,\omega)$ of the second Baire category in $\CL{C}^{\infty}$-topology of regular almost complex structures, such that for any $J\in\CL{J}_{reg}$ the space $\CL{M}_s(A,J)$ is a smooth manifold of dimension $$\dim \CL{M}_s(A,J)=2n+2I_c(A).$$ It carries a canonical orientation. For different regular almost complex structures, the moduli spaces are oriented cobordant.
\end{thm}

It follows from \cite[Proposition $2.5.1$]{MS3}) that the group of M\"{o}bius transformations $G:=PSL_2(\CC)$ acts freely on $\CL{M}_s(A,J)$ by $\phi\cdot u:=u\circ\phi^{-1}$ for every $\phi\in G$. Hence, for any $J\in\CL{J}_{reg}$ the quotient $\CL{C}_s(A,J):=\CL{M}_s(A,J)\big/ G$ is a finite dimensional manifold of dimension
$$\mathrm{dim}\ \CL{C}_s(A,J)=2n+2I_c(A)-6.$$ In particular, it is empty whenever $I_c(A)<3-n$.

In order to deal with multiply covered spheres  one can consider  moduli spaces of perturbed J-holomorphic spheres. In the closed case it was done by Y. Ruan in \cite[ Section $3$]{Ruan}. A similar construction can be found in the monograph \cite[ Chapter $8$]{MS3}. We shall consider  the Hamiltonian perturbations with compact support. For any almost complex structure $J\in\CL{J}(M,\del M,\omega)$, any Hamiltonian perturbation with compact support  $\KL{H}\in\Omega_c^1\(\mathbb{S}^2,\CL{C}^{\infty}_c(M)\)$ and any class $A\in H_2^S(M)$, denote by $\CL{M}(A,J,\KL{H})$ the set of all  $(J,\KL{H})$-holomorphic spheres in $M$, which represent the class $A$, i.e.
$$\CL{M}(A,J,\KL{H})=\{u\in\CL{C}^{\infty}(\mathbb{S}^2,M)|\ \cl{\del}_{J,\Kl{H}}(u)=0,\ [u]=A\}.$$

\begin{thm}\label{thm:moduli of perturbed spheres is a manifold}(See \cite[ Theorems $3.2.2$ and $3.2.8$]{Ruan}.)\\
There exists a subset  $\Omega_{c,\mathrm{reg}}^1\(\mathbb{S}^2,\CL{C}^{\infty}_c(M)\)\subseteq \Omega_c^1\(\mathbb{S}^2,\CL{C}^{\infty}_c(M)\)$ of the second category in the $\CL{C}^{\infty}$-topology, such that for any $\KL{H}\in\Omega_{c,\mathrm{reg}}^1\(\mathbb{S}^2,\CL{C}^{\infty}_c(M)\)$ the moduli space $\CL{M}(A,J,\KL{H})$ is a smooth manifold of dimension $$\dim\CL{M}(A,J,\KL{H})=2n+2I_c(A).$$ It carries a canonical orientation. Moreover, for a generic (in the Baire categorical sense) path $$I:=[0,1]\ni t\mapsto(J_t,\KL{H}_t)$$  the set  $$ \CL{K}\(A,(J_t,\KL{H}_t)_{t\in I}\):=\{(t,u)\in I\times\CL{C}^{\infty}(\mathbb{S}^2,\overset{\circ}{M})|\ u\in\CL{M}\(A,J_t,\KL{H}_t\)\}$$  is a smooth oriented manifold of  dimension $$\dim\ \CL{K}\(A,(J_t,\KL{H}_t)_{t\in I}\)=2n+2I_c(A)+1$$ and with boundary $$\del  \CL{K}\(A,(J_t,\KL{H}_t)_{t\in I}\)= \CL{M}(A,J_1,\KL{H}_1)\cup -\CL{M}(A,J_0,\KL{H}_0),$$ where the minus sign denotes the reversed orientation.
\end{thm}

Let us fix a finite number of pairwise distinct points $\mathbf{z}:=(z_1,\ldots,z_m)\in\(\mathbb{S}^2\)^m$. We can consider well-defined evaluation maps
$$\ev_{\mathbf{z},J}:\CL{M}_s(A,J)\to\Big(\overset{\circ}{M}\Big)^m,\;\;\;\ev_{\mathbf{z},J}(u):=(u(z_1),\ldots,u(z_m)),$$
and
$$\mathrm{ev}_{\mathbf{z},J,\Kl{H}}:\CL{M}(A,J,\KL{H})\to\Big(\overset{\circ}{M}\Big)^m,\;\;\;\ev_{\mathbf{z},J,\Kl{H}}(u):=(u(z_1),\ldots,u(z_m)).$$

\noindent Given a smooth submanifold $X\subseteq\overset{\circ}{M}$  define two moduli spaces $$\CL{M}_s(A,J;\mathbf{z},X):=\{u\in\CL{M}_s(A,J)|\ \ev_{\mathbf{z},J}(u)\in X\}$$
and
$$\CL{M}(A,J,\KL{H};\mathbf{z},X):=\{u\in\CL{M}(A,J,\KL{H})|\ \ev_{\mathbf{z},J,\Kl{H}}(u)\in X\}.$$

\begin{thm}\label{thm:constrained moduli space is a finite dim manifold}(E.g. \cite[ Theorem $3.4.1$]{MS3}.)\\
There exist  sets $\CL{J}_{\mathrm{reg}}(A,\mathbf{z},X)$ (resp. $\Omega_{c,\mathrm{reg}}^1\(\mathbb{S}^2,\CL{C}^{\infty}_c(M),A,\mathbf{z},X\)$ ) of the second category in $\CL{J}(M,\del M,\omega)$ (resp. $\Omega_c^1\(\mathbb{S}^2,\CL{C}^{\infty}_c(M)\)$), such that for any $J\in\CL{J}_{\mathrm{reg}}(A,\mathbf{z},X)$ (resp. $(J,\KL{H})\in\CL{J}(M,\del M,\omega)\times\Omega_{c,\mathrm{reg}}^1\(\Sigma,\CL{C}^{\infty}_c(M),A,\mathbf{z},X\)$ ) the moduli space $\CL{M}_s(A,J;\mathbf{z},X)$ (resp. $\CL{M}(A,J,\KL{H};\mathbf{z},X)$) is a finite dimensional smooth manifold of dimension $$\dim\CL{M}_s(A,J;\mathbf{z},X)=\dim\CL{M}(A,J,\KL{H};\mathbf{z},X)=2n+2I_c(A)-\codim X.$$
\end{thm}

\section{Genus zero Gromov-Witten invariants relative to the boundary.}
\subsection{Interior pseudocycles and oriented bordisms.}Let $X$ be a smooth compact $n$-dimensional manifold with boundary $\del X$ or more generally compact manifold with corners. We shall deal only with the simplest case of manifolds with corners, namely the product of compact manifolds with boundary. As usual, $\overset{\circ}{X}:=X\minus\del X$ denotes the interior of $X$.
\subsubsection{Interior pseudocycles.}
\begin{defn}
 An arbitrary subset $B\subseteq X$ is said to be of \textit{dimension at most} $d$ if it is contained in the image of a map $g:W\to X$, which  is defined on a manifold $W$ whose components have dimension less than or equal to $d$. In this case we write $\dim B\leq d$.
\end{defn}

\begin{defn}\label{defn: pseudocycles}
A $d$-dimensional \textit{interior pseudocycle} in $X$ is a smooth map $f:V\to X$ defined on an oriented $d$-dimensional manifold $V$, such that
\be
\item $\cl{f(V)}$ is compact  in $X$,
\item $\cl{f(V)}\subsetneq\overset{\circ}{X}$
\item $\dim\Omega_f\leq d-2,$ where $\Omega_f:=\bigcap\limits_{K\Subset V}\cl{f(V\minus K)}$ and $K\Subset V$ means that $K$ is a compact subset of $V$. This set is called the \textit{(omega) limit set} of $f$.
\ee
    Two $d$-dimensional interior pseudocycles $f_0:V_0\to X$ and $f_1:V_1\to X$ are called \textit{(interiorly) bordant} if there is a $(d+1)$-dimensional oriented manifold $W$ with boundary $\del W=V_1\cup(-V_0)$ and a smooth map $F: W\to X$ such that $$\cl{F(W)}\subsetneq\overset{\circ}{X},\;\;\;\;F|_{V_0}=f_0,\;\;\;\;F|_{V_1}=f_1,\;\;\;\;\dim\Omega_F\leq d-1.$$
\end{defn}

We have the  following elementary properties of pseudocycles.

\begin{prop}\label{prop:properties of pseudocycles}
Let $f:V\to X$ be a $d$-dimensional interior pseudocycle.
\be
\item[$(i)$] A point $x$ lies in $\Omega_f$ if and only if $x$ is the limit point of a sequence $\{f(v_n)\}_{n\in\nn}$, where $\{v_n\}_{n\in\nn}$ has no convergent subsequence.
\item[$(ii)$] The limit set $\Omega_f$ is always compact.

\item[$(iii)$] If $V$ is the interior of a compact manifold $\cl{V}$ with boundary $\del\cl{V}$ and $f$ extends to a continuous map $f:\cl{V}\to X$ then $\Omega_f=f(\del\cl{V}).$
\ee
\end{prop}

\subsubsection{Index of intersection.}
Suppose that $X$ is a smooth compact manifold  with boundary $\partial X$. Let  $e:U\to X^m$  be an interior  pseudocycle in $X^m$, for some $m\in\NN$. In addition, let  $f_i:V_i\to X$ be a smooth  singular manifold  in  $X$  for $i=1,\ldots, p$ and  let  $f_j:(V_j,\partial V_j)\to(X,\partial X)$ be a smooth  singular manifold  in  $(X,\partial X)$, for $j=p+1,\ldots, m$, where $p\in\{0,\ldots,m\}$. Then $$f:=\prod\limits_{i=1}^mf_i:\prod\limits_{i=1}^mV_i\to X^m$$
is a smooth map of manifolds with corners, such that $f\Big(\partial\big(\prod_{i=1}^mV_i\big)\Big)\subseteq\partial\(X^m\).$  In this case we would like to define the index of intersection $e\cdot_{\ff} f $ of these two "cycles". So we need some notion of transversality.
\begin{defn}\label{defn: transversale pseudocycles}
Let  $e:U\to X^m$ and $f:\prod\limits_{i=1}^mV_i\to X^m$ be two maps as above. We say that they are \textbf{transverse} if
\be
\item[$1.$] $\Omega_e\cap\cl{f\(\prod\limits_{i=1}^mV_i\)}=\varnothing$,
\item[$2.$] $e(u)=f(v)=:x\in X^m \;\;\;\then\;\;\; T_xX^m=\mathrm{Im}\ d_ue+\mathrm{Im}\ d_vf.$
\ee
We shall write classically $e\pitchfork f$.
\end{defn}

If $e\pitchfork f$ then the set
$$\Delta_{e,f}:=\left\{(u,v)\in U\times \prod_{i=1}^mV_i|\ e(u)=f(v)\right\}\subseteq U\times\prod_{i=1}^m\overset{\circ}{V_i}$$
is a compact (by property $1$ of  Definition~\ref{defn: transversale pseudocycles}\ )\  manifold of dimension $\dim U+\sum_{i=1}^m\dim V_i - m\dim X$ (by property $2$ of  Definition~\ref{defn: transversale pseudocycles}). In particular, this set is finite if $\dim U+\sum_{i=1}^m\dim V_i= m\dim X$.

\begin{thm}\label{thm:index of intersection of pseudocycles}
Let  $e:U\to X^m$ and $f:\prod\limits_{i=1}^mV_i\to X^m$ be two maps as above, such that $\dim U+\sum_{i=1}^m\dim V_i= m\dim X$.
\be
\item[$(i)$] There exists a set $\mathrm{Diff}_{reg}(e,f)\subseteq(\mathrm{Diff}(X,\partial X))^m$ of the second category such that $e\pitchfork \varphi\circ f$ for every $\varphi\in\mathrm{Diff}_{reg}(e,f)$. Here,  $\mathrm{Diff}(X,\partial X)$ denotes  the group of  diffeomorphisms  $\varphi:X\to X$, such that $\varphi|_{\partial X}=\mathrm{id}_{\partial X}$.
\item[$(ii)$] If $e\pitchfork f$ then the set $\Delta_{e,f}$ is finite. In the non-oriented case, where the base field $\FF$ is $\ZZ_2$, we define the $\ZZ_2$-valued index of intersection $e\cdot_{\ff} f$ of $e$ and $f$ as
$$e\cdot_{\ff} f=\#\Delta_{e,f}\;(\;\mathrm{mod}\ 2).$$
In the oriented case, where the base field $\FF$ is $\QQ$, we define the $\ZZ$-valued index of intersection $e\cdot_{\ff} f$ of $e$ and $f$ as
$$e\cdot_{\ff}  f=\sum\limits_{(u,v)\in\Delta_{e,f}}I_{(u,v)},$$
where $I_{(u,v)}$ is the classical local intersection number of $e(U)$ and $f(V)$ at $e(u)=f(v)$.

\item[$(iii)$]  The index of intersection $e\cdot_{\ff}  f$ depends only on the bordism classes of $e$ and of $(V_i,f_i)$  for $i=1,\ldots, m$.
\ee
\end{thm}
\begin{proof}
We consider only the oriented case. The non-oriented one is proved in a similar way.\\
The proof of statements $(i)-(ii)$ verbatim repeats that in \cite[ Lemma $6.5.5$]{MS3}. For the statement $(iii)$, suppose first that $e$ is interiorly bordant to the empty set,  and $E:W\to X^m$ is a corresponding bordism. By the standard argument on the general position in differential topology this bordism can be chosen to be transversal to $f$ in the sense of Definition~\ref{defn: transversale pseudocycles}.  Condition 1 in Definition~\ref{defn: transversale pseudocycles}  follows from the inequality  $\dim\Omega_E\leq\dim U-1$ that implies
$$\dim\Omega_E+\sum_{i=1}^m\dim V_i-m\dim X\leq -1.$$
Hence, the set
$$\Delta_{E,f}=\left\{(w,v)\in W\times \prod_{i=1}^mV_i|\ E(w)=f(v)\right\}\subseteq W\times\prod_{i=1}^m\overset{\circ}{V_i}$$
is a compact oriented $1$-dimensional manifold with boundary $\partial\Delta_{E,f}=\Delta_{e,f}$. Thus $e\cdot_{\ff}  f=0$. Secondly, suppose that $[V_1,f_1]=0$ with a corresponding bordism $F_1:W_1\to X$ and  put $$
F:=F_1\times f_2\times\cdots\times f_m:W_1\times V_2\times\cdots\times V_m\to X^m.
$$
Then $F$  is a smooth map of manifolds with corners, such that $$F\Big(\partial\big(W_1\times\prod_{i=2}^mV_i\big)\Big)\subseteq\partial\(X^m\),$$ and, as above,  we can choose this bordism  to be transversal to $e$ in the sense of Definition~\ref{defn: transversale pseudocycles}.  Condition 1 in  Definition~\ref{defn: transversale pseudocycles}  follows from the inequality
$$
\dim\Omega_e+\dim W_1+\sum_{i=2}^m\dim V_i-m\dim X\leq -1.
$$
Hence, the set
$$
\Delta_{e,F}=\left\{(u,v)\in U\times \Big(W_1\times\prod_{i=2}^mV_i\Big)|\ e(w)=F(v)\right\}\subseteq U\times W_1\times\prod_{i=2}^m\overset{\circ}{V_i} $$
is a compact oriented $1$-dimensional manifold with boundary $\partial\Delta_{e,F}=\Delta_{e,f}$. Thus $e\cdot_{\ff}  f=0$.  The same argument shows that $e\cdot_{\ff}  f$ does not depend on $(V_i,f_i)$  for $i=2,\ldots, m$.
\end{proof}
As a consequence, every interior pseudocycle $e:U\to X^m$ determines a well defined $m$-linear map
\beq\label{equation:itersection homomorphism til{Psi}_e}
\hat{\Psi}_{e,p}:\Omega_*(X;\FF)^{\times p}\times\Omega_*(X,\del X;\FF)^{\times (m-p)}\to\FF
\eeq
by $$\hat{\Psi}_{e,p}([V_1,f_1],\ldots, [V_m,f_m])=e\cdot_{\ff}(f_1\times\dots\times f_m),$$
where $\Omega_*(X;\FF), \Omega_*(X,\del X;\FF)$ are the classical Thom bordism groups -- see \cite{CF} and \cite{Stong}. Note that the homomorphism $\hat{\Psi}_{e,p}$ depends only on the bordism class of the  interior pseudocycle $e$.

Recall that the evaluation homomorphism
$$\mu^{p,m}:\Omega_*(X;\FF)^{\times p}\times\Omega_*(X,\del X;\FF)^{\times (m-p)} \to  H_*(X;\FF)^{\times p}\times H_*(X,\del X;\FF)^{\times (m-p)}$$
given by
$$\mu^{p,m}([V_1,f_1],\ldots, [V_m,f_m])=\( \((f_i)_*([V_i])\)_{i=1}^p, \((f_j)_*([V_j,\del V_j]\)_{j=p+1}^m \)$$
is an epimorphism. In particular, the homomorphism $\hat{\Psi}_{e,p}$ descends to a well defined homomorphism  \beq\label{equation:itersection homomorphism Psi_e}
    \Psi_{e,p}:H_*(X;\FF)^{\times p}\times H_*(X,\del X;\FF)^{\times (m-p)}\to\FF,
\eeq
which also depends only on the bordism class of the  interior pseudocycle $e$.

\subsection{Gromov-Witten pseudocycle.}
Let us fix a finite number of pairwise distinct marked points $\mathbf{z}:=(z_1,\ldots,z_m)\in\(\mathbb{S}^2\)^m.$ In the following constructions we shall consider mainly the case of $m\in\{3,4\}$.

\subsubsection{Case \textrm{I} - non-perturbed.}\label{subsection: conditions Am}
Let  $A\in H_2^S(M)$ and $J\in\CL{J}(M,\del M,\omega)$.  We shall compactify the image
$$\mathrm{ev}_{\mathbf{z},J}(\CL{M}_s(A,J))\subseteq\Big(\overset{\circ}{M}\Big)^m$$ by adding all possible limits of points $\mathrm{ev}_{\mathbf{z},J}(u_n)$, where $(u_n)_{n\in\nn}$ is any sequence in $\CL{M}_s(A,J)$. By the Gromov compactness theorem the  boundary of  $\mathrm{ev}_{\mathbf{z},J}(\CL{M}_s(A,J))$ will be formed by the $\mathrm{ev}_{\mathbf{z},J}$-images of different cusp (reducible) curves. By the convexity it follows that the compactified space $\cl{\mathrm{ev}_{\mathbf{z},J}(\CL{M}_s(A,J))}$  lies in $\Big(\overset{\circ}{M}\Big)^m$.  Thus we can state the following theorem, which are completely analogous to the closed case. Recall that a class $B\in H_2^S(M)$ is called $J$-effective if it can be represented by a $J$-holomorphic sphere. Assume that  $A\in H_2^S(M)$  satisfies  the following conditions.
\be
\item[$(A_3)$] In the case $m=3$, i.e. the case of three marked points, we require that every $J$-effective  class $B\in H_2^S(M)$  has Chern number  $c_1(B)\geq0$. Moreover, the class $A$ is not a multiple $A=kB$ of a $J$-effective  class $B\in H_2^S(M)$  with  $k>1$ and $c_1(B)=0$.

\item[$(A_4)$] In the case $m=4$, i.e. the case of four marked points, we require that every $J$-effective class $B\in H_2^S(M)$ has Chern number $c_1(B)\geq 1$. Moreover, the class $A$ is not a multiple $A=kB$ of a $J$-effective  class $B\in H_2^S(M)$  with  $k>1$ and $c_1(B)=1$.
\ee

\begin{thm}\label{thm:GW-pseudocycle in the simple case}(E.g. \cite[ Theorem $5.4.1$]{MS2}.)\\
Let  $(M,\omega)$ be a $2n$-dimensional compact convex semi-positive  symplectic manifold. Fix a class  $A\in H_2^S(M)\minus\{0\}$ and  fix  pairwise distinct marked points
$$
\mathbf{z}:=(z_1,\ldots,z_m)\in\(\mathbb{S}^2\)^m,
$$
where $m\in\{3,4\}$. Then there exists a set
$$
\CL{J}_{reg}(M,\partial M,\omega,A,\mathbf{z})\subseteq\CL{J}(M,\partial M,\omega)
$$
of the second category, such that
\be
\item[$(i)$] $\CL{J}_{reg}(M,\partial M,\omega,A,\mathbf{z})\subseteq\CL{J}_{reg}(A)$,
    see the notation of Theorem~\ref{thm:moduli space of simple curves is a manifold},

\item[$(ii)$] if $J\in\CL{J}_{reg}(M,\partial M,\omega,A,\mathbf{z})$ and the class $A$ satisfies $(A_m)$, then
    $$
    \mathrm{ev}_{\mathbf{z},J}:\CL{M}_s(A,J)\to M^m
    $$
    is an interior (oriented) pseudocycle of dimension $2n+2c_1(A)$. Its bordism class is independent of $J$ and $\mathbf{z}$.
\ee
\end{thm}

\subsubsection{Case \textrm{II} - perturbed.}
We should treat the case, where $A\in H_2^S(M)$ does not satisfy conditions $(A_3)$, $(A_4)$. There are several ways of dealing with this problem. One can use the argument by Y. Ruan in \cite[ Section $3$]{Ruan}, where he considers the perturbed Cauchy-Riemann equation $\cl{\del}_{J,\Kl{H}}(u)=0$ that already has no multiply covered solutions. See also \cite[ Section $9.1$]{MS2} or \cite[ Chapter $8$]{MS3}.  Like in \textit{Case \textrm{I}}, we compactify the image
$\mathrm{ev}_{\mathbf{z},J,\Kl{H}}(\CL{M}(A,J,\KL{H}))$ by adding all possible limits of points $\mathrm{ev}_{\mathbf{z},J,\Kl{H}}(u_n)$, where $(u_n)_{n\in\nn}$ is any sequence in $\CL{M}(A,J,\KL{H})$. It follows from the convexity  that the compactified space $\cl{\mathrm{ev}_{\mathbf{z},J,\Kl{H}}(\CL{M}(A,J,\KL{H}))}$  lies in $\Big(\overset{\circ}{M}\Big)^m$.

\begin{thm}\label{thm:GW-pseudocycle in the perturbed case}(E.g. \cite[ Section $3$]{Ruan} or \cite[ Items $8.5.1 -  8.5.4$]{MS3}.)\\
Let  $(M,\omega)$ be a $2n$-dimensional compact convex semi-positive  symplectic manifold. Fix a class  $A\in H_2^S(M)\minus\{0\}$ and  fix  pairwise distinct marked points
$$
\mathbf{z}:=(z_1,\ldots,z_m)\in\(\mathbb{S}^2\)^m,
$$
where $m\in\{3,4\}$. Then there exists a set
$$
\CL{JH}_{reg}(M,\partial M,\omega,A,\mathbf{z})\subseteq\CL{J}(M,\partial M,\omega)\times\Omega_c^1\(\mathbb{S}^2, C^{\infty}_c(M)\)
$$
of the second category such that
\be
\item[$(i)$] for every $(J,\KL{H})\in\CL{JH}_{reg}(M,\partial M,\omega,A,\mathbf{z})$, we have that  $\KL{H}\in\Omega_{c,\mathrm{reg}}^1\(\mathbb{S}^2, C^{\infty}_c(M)\)$,
    see the notation of Theorem~\ref{thm:moduli of perturbed spheres is a manifold},

\item[$(ii)$] if $(J,\KL{H})\in\CL{JH}_{reg}(M,\partial M,\omega,A,\mathbf{z})$
    $$
    \mathrm{ev}_{\mathbf{z},J,\Kl{H}}:\CL{M}(A,J,\KL{H}) \to M^m
    $$
    is an interior (oriented) pseudocycle of dimension $2n+2c_1(A)$. Its bordism class is independent of $(J,\KL{H})$ and $\mathbf{z}$.
\ee
\end{thm}

\subsection{Invariants $GW_{A,p,m}$.}
Following \cite{Ruan}, \cite{MS2}, \cite{MS3} we shall define  two types of genus zero Gromov-Witten invariants. But first, let us recall the definition of intersection products from  algebraic topology. Recall that  $\FF$  is  either $\ZZ_2$ or $\QQ$.
\begin{defn}\label{defn: Intersection products}
\mbox\\
Homomorphisms
\bea\label{equation: intersection products in homology}
&\bullet_1:H_i(M;\FF)\otimes H_j(M;\FF)\to H_{i+j-2n}(M;\FF)\\
&\bullet_2:H_i(M;\FF)\otimes H_j(M,\partial M;\FF)\to H_{i+j-2n}(M;\FF)\\
&\bullet_3:H_i(M,\partial M;\FF)\otimes H_j(M,\partial M;\FF)\to H_{i+j-2n}(M,\partial M;\FF)
\eea
given by
\bea\label{equation: formula of intersection products in homology}
&a\bullet_1 b:=\mathrm{PLD}_2\(\mathrm{PLD}_2^{-1}(b)\cup \mathrm{PLD}_2^{-1}(a)\)\\
&a\bullet_2 b:=\mathrm{PLD}_2\(\mathrm{PLD}_2^{-1}(b)\cup \mathrm{PLD}_1^{-1}(a)\)\\
&a\bullet_3 b:=\mathrm{PLD}_1\(\mathrm{PLD}_1^{-1}(b)\cup \mathrm{PLD}_1^{-1}(a)\)
\eea
are called the intersection products in homology. Here
\bea\label{equation:  Poincare-Lefschetz duality isomorphisms}
&\mathrm{PLD}_1:H^j(M;\FF)\to H_{2n -j}(M,\partial M;\FF)\\
&\mathrm{PLD}_2:H^j(M,\partial M;\FF)\to H_{2n -j}(M;\FF)
\eea
are the Poincar\'{e}-Lefschetz duality isomorphisms given by $$\mathrm{PLD}_i(\alpha):=\alpha\cap[M,\partial M],\;\;i=1,2,$$ where $[M,\partial M]$ is the positive generator of $H_{2n}(M,\partial M;\FF)\cong\FF$ --  the relative fundamental class.
\end{defn}

\begin{thm-defn}\label{thm:definition of GW_[m]}(E.g.\cite{Ruan}; \cite{RT}; \cite[ Chapter $7$]{MS2}; \cite[ Chapter  $7$]{MS3}.)\\
Let  $(M,\omega)$ be a $2n$-dimensional compact convex semi-positive  symplectic manifold, $m\in\{3,4\}$ and $p\in\{0,1,\ldots,m\}$. Fix  pairwise distinct marked points $$\mathbf{z}:=(z_1,\ldots,z_m)\in\(\mathbb{S}^2\)^m.$$

\be
\item[$(i)$] Let $A\in H_2^S(M)\minus\{0\}$ and $(J,\KL{H})\in\CL{JH}_{reg}(M,\partial M,\omega,A,\mathbf{z})$. Then the $m$-linear map
    $$GW_{A,p,m}: H_*(M;\FF)^{\times p}\times H_*(M,\partial M;\FF)^{\times (m-p)}\to\FF$$
     given by
     $$GW_{A,p,m}(a_1,\ldots,a_m):=\Psi_{\ev_{\mathbf{z},J,\Kl{H}},p}(a_1,\ldots,a_m)$$
      is well-defined and is independent of the pair $(J,\KL{H})\in\CL{JH}_{reg}(M,\partial M,\omega,A,\mathbf{z})$ and of the tuple $\mathbf{z}$ of fixed marked points.

\item[$(ii)$] If $A\in H_2^S(M)\minus\{0\}$ satisfies $(A_m)$ (see Section \ref{subsection: conditions Am}) then a pair $(J,\KL{H}=0)$ with  $J\in\CL{J}_{reg}(M,\partial M,\omega,A,\mathbf{z})$ belongs to $\CL{JH}_{reg}(M,\partial M,\omega,A,\mathbf{z})$.

\item[$(iii)$]The Gromov-Witten invariant $GW_{A,p,m}$  depends only on the semi-positive deformation class of $\omega$.
\ee
\end{thm-defn}
\begin{proof}
Statements $(i)$ and $(ii)$ follow directly from Theorems~\ref{thm:GW-pseudocycle in the simple case}, \ref{thm:GW-pseudocycle in the perturbed case}. The proof of $(iii)$ repeats verbatim the proof of \cite[Proposition $2.3$]{RT}.
\end{proof}

Since for $A=0$ the above evaluation maps are not  interior pseudocycles, we  define the Gromov-Witten invariants for these classes as follows.

\begin{defn}\label{defn: GW for A in ker h2}
Let  $(M,\omega)$ be a $2n$-dimensional compact convex semi-positive  symplectic manifold. Let $m\in\{3,4\}$ and $p\in\{0,1,\ldots,m\}$. If $$\sum_{i=1}^m\deg(a_i)\neq2n(m-1),$$ define $$GW_{0,p,m}(a_1,\ldots, a_m): =0.$$ Otherwise, define
\beq\label{equation:  GW[p,m]  for A in ker h2 }
\hspace{-10mm}GW_{0,p,m}(a_1,\ldots, a_m): =
            \begin{cases}
            0&\text{ if }\ p=0  ,\\
            (\dots((a_1\bullet_2 a_2)\bullet_2 a_3)\dots)\bullet_2 a_m  &\text{ if }\ p=1,\\
            (\dots((a_1\bullet_1 a_2)\bullet_2 a_3)\dots)\bullet_2 a_m&\text{ if }\ p=2,\\
            (\dots((a_1\bullet_1 a_2)\bullet_1 a_3)\dots)\bullet_2 a_m&\text{ if }\ p=3,\\
            a_1\bullet_1 a_2\bullet_1 a_3\bullet_1 a_4&\text{ if }\ p=4.
            \end{cases}
\eeq
\end{defn}
We see that this definition correlates with the closed case and reflects the behavior of the relative intersection products.\\

The next properties follow immediately from the definition of the Gromov-Witten invariants, see \cite[ Chapter $7$]{MS3}; \cite{Ruan}; \cite{RT}.
\begin{prop}\label{prop: properties of GW invariants}
Let  $(M,\omega)$ be a $2n$-dimensional compact convex semi-positive  symplectic manifold.
\be
\item[$(i)$]Let $m\in\{3,4\}$ and $p\in\{0,1,\ldots,m\}$. Following \cite[ Section $7.5$]{MS3}  for each permutation $\sigma\in S_m$ and each monomial  $$a:=(a_1,\ldots, a_m)\in H_*(M;\FF)^{\times p}\times H_*(M,\partial M;\FF)^{\times (m-p)}$$ denote by $$\ve:=\varepsilon(\sigma;a):=(-1)^{\#\{i<j|\sigma(i)>\sigma(j), \deg(a_i)\cdot\deg(a_j)\in2\zz+1\}}$$ the sign of the induced permutation on the classes of odd degree.  Then
    \bea\label{equation:skewsymmetry of GW}
     GW_{A,p,m}(\sigma_*a)&=\varepsilon GW_{A,p,m}(a),
    \eea
     where $\sigma_*a:=a_{\sigma(1)},\ldots, a_{\sigma(m)}$.
\item[$(ii)$]If $m=3, p\in\{0,1,2\}$ and  $A\in H_2^S(M)\minus\{0\}$,  then
    \beq\label{equation: GW with fundamental class for m=3}
    GW_{A,p,3}(a_1, a_2,  [M,\partial M])=0.
    \eeq
\item[$(iii)$] Let $j_*: H_*(M;\FF)\to H_*(M,\partial M;\FF)$ be the natural homomorphism, induced by the inclusion $(M,\varnothing)\hookrightarrow(M,\partial M)$. Let $m\in\{3,4\}$, $p\in\{1,\ldots,m\}$ and $A\in H_2^S(M)\minus\{0\}$. Then
    \beq\label{equation: GW for j(a1)=0}
GW_{A,p,m}(a_1,\ldots, a_m)=0,
\eeq
for any classes  $\{a_i\}_{i=1}^m$, such that $j_*(a_1)=0$.
\ee
\end{prop}

\subsection{Calculation of Gromov-Witten invariants}
In the following examples we shall calculate the Gromov-Witten invariants over $\FF=\QQ$ of the symplectic blow-up $(\til{X}\cong X\#\cl{\CC\PP^2}, \sigma_{\delta})$ of an interior point in some  compact convex semi-positive  symplectic manifolds $(X,\sigma)$, i.e. the symplectic blow up of size $\delta$ that corresponds to a symplectic embedding
$$
\psi:\(\BB^{4}(\delta)\subseteq\CC^2, \omega_0\)\to(X,\sigma),
$$
see \cite[ Section $7.1$]{MS} or \cite[ Section $9.3$]{MS3}  for the precise definition of the symplectic blow-up $(\til{X}, \sigma_{\delta})$.

\begin{ex}\label{example: GW of blowup of a 4-ball}
\mbox{}\\ Let  $(\til{\BB^4} , \omega_{\delta})$ be the symplectic blow-up at zero of the standard unit $4$-ball in $(\CC^2,\omega_0)$ of a  small size $0<\delta<<1$. It is a  $4$-dimensional compact convex symplectic manifold equipped  with an  $\omega$-compatible almost complex structure $\til{J_0}$, where $J_0$ is the standard complex structure (multiplication by $\sqrt{-1}$) on $\CC^2$. Let $E\in H_2(\til{\BB^4};\CC)$ be the class of the exceptional divisor. It can be realized by a $\til{J_0}$-holomorphic embedding  $\iota:\CC\PP^1\to \til{\BB^4}$, i.e.  $E=[\iota(\CC\PP^1)]=\iota_*[\CC\PP^1]$ and $\omega_{\delta}(E)=\delta$. Set $C:=\iota(\CC\PP^1)$. A standard calculation shows that $c_1(E)=1$.
In particular, the minimal Chern number $N$ of $\til{\BB^4}$  (see Section 1.2) equals $1$ and the class $E$ satisfies condition $(A_3)$, (see Section \ref{subsection: conditions Am}). Note that
\beq\label{equation: absolute homology of BCP}
H_i(\til{\BB^4};\QQ)\cong
\begin{cases}
\mathrm{Span}_{\qq}([pt])\cong\QQ&,i=0,\\
0&,i=1,\\
\mathrm{Span}_{\qq}(E)\cong\QQ&,i=2,\\
0&,i=3,\\
0&,i=4.
\end{cases}
\eeq
\beq\label{equation: relative homology of BCP}
H_i(\til{\BB^4},\partial \til{\BB^4};\QQ)\cong
\begin{cases}
0&,i=0,\\
0&,i=1,\\
\mathrm{Span}_{\qq}(E^{\vee})\cong\QQ&,i=2,\\
0&,i=3,\\
\mathrm{Span}_{\qq}([\til{\BB^4},\partial \til{\BB^4}])\cong\QQ&,i=4.
\end{cases}
\eeq
\mbox{}\\
By \cite[ Section $3.3$ and Lemma $7.1.8$]{MS3}, we can use the almost complex structure $\til{J_0}$ to calculate the invariants $GW_{E,p,3}$. By the non-negativity of intersection in 4-dimensional almost complex manifolds, cf. \cite[ Theorem $2.6.3$]{MS3},  the moduli space $\CL{M}(E,\til{J_0})$ consists of a single $\til{J_0}$-holomorphic curve up to a reparametrization, namely the exceptional sphere itself. Take $\{[pt], E\}$ and $ \{E^{\vee},[\til{\BB^4},\partial \til{\BB^4}]\}$ to be homogeneous bases  for $H_*(\til{\BB^4};\CC)$ and $H_*(\til{\BB^4},\partial \til{\BB^4};\CC)$ respectively. Then $$j_*([pt])=0, E\bullet_1 E=-1, E\bullet_2 E^{\vee}=1,$$ where  $j_*: H_*(\til{\BB^4};\QQ)\to H_*(\til{\BB^4},\partial \til{\BB^4};\QQ)$ is the natural homomorphism induced by the inclusion. Thus, the only non-zero Gromov-Witten invariants for the class $E$ and the basic homology classes are
\bea\label{equation: GW of  BCP }
&GW_{0,1,3}(E, E^{\vee}, [\til{\BB^4},\partial \til{\BB^4}])=(E\bullet_2E^{\vee})\bullet_2 [\til{\BB^4},\partial \til{\BB^4} ]=1,\\
&GW_{0,2,3}(E, E, [\til{\BB^4},\partial \til{\BB^4}])=(E\bullet_1E)\bullet_2 [\til{\BB^4},\partial \til{\BB^4}]=-1,\\
&GW_{E,0,3}(E^{\vee}, E^{\vee}, E^{\vee})=(E\bullet_2E^{\vee})^3=1,\\
&GW_{E,1,3}(E, E^{\vee}, E^{\vee})=(E\bullet_2E^{\vee})^2(E\bullet_1E)=-1,\\
&GW_{E,2,3}(E, E, E^{\vee})=(E\bullet_2E^{\vee})(E\bullet_1E)^2=1,\\
&GW_{E,3,3}(E, E, E)=(E\bullet_1E)^3=-1.
\eea
By \cite[Theorem $2.6.4$]{MS3}  and \cite[Proposition $2.3$]{H-S}, we have that for all $d\in\ZZ\minus\{0,1\}$ the moduli space $\CL{M}_s(dE,J)$ is empty for any regular $J\in\CL{J}_{reg}(dE)$. Hence $GW_{dE,p,3}=0$ for all $p\in\{0,\ldots,3 \}$ and for all $d\in\ZZ\minus\{0,1\}$.
\end{ex}

\begin{ex}\label{example: GW of blowup of a  cotangent disk bundle of the surface}
\mbox{}\\ Let $\Sigma$ be a smooth closed surface of genus $g\geq 0$ and $G$ -- a fixed Riemannian metric on $\Sigma$. Consider the  cotangent unit disk bundle $X:=\DD T^*\Sigma$ w.r.t. $G$  equipped with the canonical symplectic form $\omega_{can}:=d\lambda_{can}$. The metric $G$ induces a horizontal-vertical splitting of $TT^*\Sigma\overset{G}{\cong}T\Sigma\oplus T^*\Sigma$ and as a result, it induces an $\omega_{can}$-compatible almost complex structure $J_G$ on $X$, which in the horizontal-vertical splitting takes the form
$$J_G=
\begin{pmatrix}
0&-\Id\\\Id&0
\end{pmatrix}.
$$
Let $(\til{X}, \omega_{\delta})$  be the symplectic blow-up of $X$ of a sufficiently small size $0<\delta<<1$, so that the corresponding symplectic $\delta$-ball lies in $\overset{\circ}{X}$
and  doesn't intersect the zero section $\Sigma$ of $X$.
Then $(\til{X}, \omega_{\delta})$ is a  $4$-dimensional compact convex symplectic manifold equipped  with an $\omega_{\delta}$-compatible almost complex structure $\til{J_G}$. Let $E\in H_2(\til{X};\QQ)$ be the class of the exceptional divisor. It is realized by a $\til{J_G}$-holomorphic embedding  $\iota:\CC\PP^1\to \til{X}$ and $\omega_{\delta}(E)=\delta$.
The minimal Chern number $N$  of $\til{X}$ equals $1$, and the class $E$ satisfies condition $(A_3)$. We also have
\beq\label{equation: absolute homology of DT*SCP}
H_i(\til{X};\QQ)\cong
\begin{cases}
\mathrm{Span}_{\qq}([pt])\cong\QQ&,i=0,\\
H_1(\Sigma;\QQ)=\begin{cases}0&, g=0\\ \mathrm{Span}_{\qq}([a_1],\ldots, [a_g], [b_1]\ldots, [b_g])\cong\QQ^{2g}&, g\geq 1\end{cases}&,i=1,\\
H_2(\Sigma;\QQ)\oplus H_2(\CC\PP^1;\QQ)\cong\mathrm{Span}_{\qq}([\Sigma])\oplus\mathrm{Span}_{\qq}(E)\cong\QQ^2&,i=2,\\
0&,i=3,\\
0&,i=4.
\end{cases}
\eeq

Denote by $[F]\in H_2(X,\partial X;\QQ)$ the class of the fiber $F\cong\DD^2$ of the fibration $X=\DD T^*\Sigma\to\Sigma$. Hence,  $[\Sigma]^{\vee}=[F]$,  $[a_i]^{\vee}=[X|_{b_i}]$, $[b_i]^{\vee}=-[X|_{a_i}]$ for $i=1,\ldots, g$. It follows that
\beq\label{equation: relative homology of  DT*SCP }
H_i(\til{X},\partial \til{X};\QQ)\cong
\begin{cases}
0&,i=0,\\
0&,i=1,\\
\mathrm{Span}_{\qq}([F])\oplus\mathrm{Span}_{\qq}(E^{\vee})\cong\QQ^2&,i=2,\\
\begin{cases}0&, g=0\\ \mathrm{Span}_{\qq}([a_1]^{\vee},\ldots, [a_g]^{\vee}, [b_1]^{\vee}\ldots, [b_g]^{\vee})\cong\QQ^{2g}&, g\geq 1\end{cases}&,i=3,\\
\mathrm{Span}_{\qq}([\til{X},\partial \til{X}])\cong\QQ&,i=4.
\end{cases}
\eeq
In addition, we have
\bea\label{equation: intersections of the basic classes in DT*SCP }
&j_*([pt])=0,\\
&E\bullet_1 E=-1,\;\; [\Sigma]\bullet_1[\Sigma]=2-2g,\;\; E\bullet_2 E^{\vee}=1,\\
&[\Sigma]\bullet_1 E=[F]\bullet_2 E=[a_i]\bullet_2 E=[b_i]\bullet_2 E=0,\\
&[a_i]\bullet_1[b_j]= [a_i]\bullet_1[a_j]=[b_i]\bullet_1[b_j]=[a_i]\bullet_2[b_j]^{\vee}=[b_i]\bullet_2[a_j]^{\vee}=0,\\
& [a_i]\bullet_2[a_j]^{\vee}=[b_i]\bullet_2[b_j]^{\vee}=\delta_{ij}.
\eea
By \cite[ Section $3.3$ and Lemma $7.1.8$]{MS3}, we can use the almost complex structure $\til{J_G}$ to calculate the invariants $GW_{E,p,3}$. By the non-negativity of intersection in 4-dimensional almost complex manifolds, cf. \cite[ Theorem $2.6.3$]{MS3}, we have that the moduli space $\CL{M}(E,\til{J_G})$ consists of a single $\til{J_G}$-holomorphic curve up to a reparametrization, namely, the exceptional sphere itself.  By \cite[ Theorem $2.6.4$]{MS3}  and \cite[Proposition $2.3$]{H-S}, we have that for all $d\in\ZZ\minus\{0,1\}$ the moduli space $\CL{M}_s(dE,J)$ is empty for any regular $J\in\CL{J}_{reg}(dE)$.  As a result, the only non-zero invariants on the basic classes are

\hspace{-5mm}\bea\label{equation: GW of  DTSCP }
&GW_{0,1,3}([a_i], [a_j]^{\vee}, [\til{X},\partial \til{X}])=[a_i]\bullet_2[a_j]^{\vee}=\delta_{ij},\\
&GW_{0,1,3}([b_i], [b_j]^{\vee}, [\til{X},\partial \til{X}])=[b_i]\bullet_2[b_j]^{\vee}=\delta_{ij},\\
&GW_{0,1,3}([\Sigma], [a_i]^{\vee}, [b_j]^{\vee}])=([\Sigma]\bullet_2[a_i]^{\vee})\bullet_2 [b_j]^{\vee} =[b_i]\bullet_2[b_j]^{\vee}=\delta_{ij},\\
&GW_{0,2,3}([\Sigma], [\Sigma],  [\til{X},\partial \til{X}])=[\Sigma]\bullet_1[\Sigma]=2-2g,\\
&GW_{0,2,3}(E, E,  [\til{X},\partial \til{X}])=E\bullet_1E=-1,\\
&GW_{E,0,3}(E^{\vee}, E^{\vee}, E^{\vee})=(E\bullet_2E^{\vee})^3=1,\\
&GW_{E,1,3}(E, E^{\vee}, E^{\vee})=(E\bullet_2E^{\vee})^2(E\bullet_1E)=-1,\\
&GW_{E,2,3}(E, E, E^{\vee})=(E\bullet_2E^{\vee})(E\bullet_1E)^2=1,\\
&GW_{E,3,3}(E, E, E)=(E\bullet_1E)^3=-1.
\eea
\end{ex}

\section{Quantum homology}
Let  $(M,\omega)$ be a $2n$-dimensional compact convex semi-positive  symplectic manifold. Following  \cite[ Chapter $11$]{MS3} we use the Gromov-Witten invariants $GW_{A,p,m}$ to deform the classical intersection products $\bullet_i, i=1,2,3$; cf. Definition~\ref{defn: Intersection products}. These deformations give rise to new ring structures on the absolute ( that is $H_*(M)$) and the relative (that is $H_*(M,\partial M)$) homology algebras of $(M,\omega)$. From now on we shall consider the following Novikov ring $\Lambda$. Let
\beq\label{equation: Gamma group of spherical classes}
\Gamma:=\Gamma(M,\omega):=\frac{H_2^S(M)}{\ker( c_1)\cap \ker(\omega)}
\eeq
and let
\beq\label{equation: G group of periods}
G:=G(M,\omega):=\frac12\omega\(H_2^S(M)\)\subseteq\RR
\eeq
be the subgroup of half-periods of the symplectic form $\omega$ on spherical homology classes. Recall that $\FF$ denotes a base field, which in our case will be either $\QQ$ or $\ZZ_2$. Let $s$ be a formal variable. Define the field $\KK_G$ of generalized Laurent series in  $s$ over $\FF$ of the form
\beq\label{equation:  field KG of generalized Laurent series}
\KK_G:=\left\{f(s)=\sum\limits_{\alpha\in G}z_{\alpha}s^{\alpha}, z_{\alpha}\in\FF|\ \#\{\alpha>c|z_{\alpha}\neq 0\}<\infty,\ \forall c\in\RR \right\}
\eeq

\begin{defn}\label{defn: Novikov ring Lambda}
Let $q$ be a formal variable. The \textbf{Novikov ring $\Lambda:=\Lambda_G$} is the ring of polynomials in $q,q^{-1}$ with coefficients in the field $\KK_G$, i.e.
\beq\label{equation:  Novikov ring}
\Lambda:=\Lambda_G:=\KK_G[q,q^{-1}].
\eeq
We equip the ring $\Lambda_G$ with the structure of a graded ring by setting $\deg(s)=0$ and $\deg(q)=1$.  We shall denote by $\Lambda_k$ the set of  elements of $\Lambda$ of degree $k$. Note that  $\Lambda_0=\KK_G$.
\end{defn}

\subsection{Quantum homology}\label{subsection: QH}
The  quantum homology  $QH_*(M)$ and the relative  quantum homology  $QH_*(M,\partial M)$ are defined as follows. As  modules, they are graded modules over $\Lambda$ defined by
$$QH_*(M):=H_*(M;\FF)\otimes_{\ff}\Lambda,\;\;QH_*(M,\partial M):=H_*(M,\partial M;\FF)\otimes_{\ff}\Lambda.$$ A grading on both modules is given by $$\deg(a\otimes zs^{\alpha}q^m)=\deg(a)+m,$$  and for $k\in\ZZ$ the degree-$k$ parts of $QH_*(M)$ and of  $QH_*(M,\partial M)$ are given by
\bean&QH_k(M):=\bigoplus_iH_i(M;\FF)\otimes_{\ff}\Lambda_{k-i},\\
&QH_k(M,\partial M):=\bigoplus_iH_i(M,\partial M;\FF)\otimes_{\ff}\Lambda_{k-i}
\eea
respectively.  \par Next, we define the quantum products $\ast_i, i=1,2,3$, which are the deformations of  the classical intersection products $\bullet_i, i=1,2,3$. Choose a homogeneous basis  $\{e_k\}_{k=1}^{d}$ of $H_*(M;\FF)$, such that $e_1=[pt]\in H_0(M;\FF)$. Let  $\{e^{\vee}_k\}_{k=1}^{d}$ be the dual homogeneous basis of $H_*(M,\partial M;\FF)$ defined by $$\la e_i, e^{\vee}_j \ra=\delta_{ij},$$ where $\la\cdot,\cdot\ra$ is the Kronecker pairing.

\begin{defn}\label{defn: quantum products}
Let $A\in H_2^S(M)$ and let $[A]\in\Gamma$ be the image of $A$ in $\Gamma$.
\item[$(i)$] A bilinear homomorphism of $\Lambda$-modules $$\ast_1:QH_*(M)\times QH_*(M)\to QH_*(M)$$ is given as follows.
If $a\in H_i(M;\FF)$ and $b\in H_j(M;\FF)$ then
$$QH_{i+j-2n}(M)\ni a\ast_1 b:=\sum_{[A]\in\Gamma}(a\ast_1 b)_{[A]}\otimes s^{-\omega(A)}q^{-2c_1(A)},$$
where $(a\ast_1 b)_{[A]}\in H_{i+j-2n+2c_1(A)}(M;\FF)$ is given by $$(a\ast_1 b)_{[A]}=\sum_{i=1}^d\sum_{A'\in[A]}GW_{A',2,3}(a,b,e^{\vee}_i)e_i.$$
We extend this definition by $\Lambda$-linearity to the whole $QH_*(M)\times QH_*(M)$.

\item[$(ii)$] A bilinear homomorphism of $\Lambda$-modules $$\ast_2:QH_*(M)\times QH_*(M,\partial M)\to QH_*(M)$$ is given as follows.
If $a\in H_i(M;\FF)$ and $b\in H_j(M,\partial M;\FF)$ then
$$QH_{i+j-2n}(M)\ni a\ast_2 b:=\sum_{[A]\in\Gamma}(a\ast_2 b)_{[A]}\otimes s^{-\omega(A)}q^{-2c_1(A)},$$
where $(a\ast_2 b)_{[A]}\in H_{i+j-2n+2c_1(A)}(M;\FF)$ is given by $$(a\ast_2 b)_{[A]}=\sum_{i=1}^d\sum_{A'\in[A]}GW_{A',1,3}(a,b,e^{\vee}_i)e_i.$$
We extend this definition by $\Lambda$-linearity to the whole $QH_*(M)\times QH_*(M,\partial M)$.

\item[$(iii)$] A bilinear homomorphism of $\Lambda$-modules $$\ast_3:QH_*(M,\partial M)\times QH_*(M,\partial M)\to QH_*(M,\partial M)$$ is given as follows. If $a\in H_i(M,\partial M;\FF)$ and $b\in H_j(M,\partial M;\FF)$ then
    $$QH_{i+j-2n}(M,\partial M)\ni a\ast_3 b:=\sum_{[A]\in\Gamma}(a\ast_3 b)_{[A]}\otimes s^{-\omega(A)}q^{-2c_1(A)},$$
    where $(a\ast_3 b)_{[A]}\in H_{i+j-2n+2c_1(A)}(M,\partial M;\FF)$ is given by $$(a\ast_3 b)_{[A]}=\sum_{i=1}^d\sum_{A'\in[A]}GW_{A',1,3}(a,b,e_i)e^{\vee}_i.$$
We extend this definition by $\Lambda$-linearity to the whole $QH_*(M,\partial M)\times QH_*(M,\partial M)$.
\end{defn}
\begin{rem}
By the Gromov compactness theorem, the sums $\displaystyle \sum_{A'\in[A]}GW_{A',2,3}(a,b,e^{\vee}_i)$, $\displaystyle\sum_{A'\in[A]}GW_{A',1,3}(a,b,e^{\vee}_i)$, $\displaystyle\sum_{A'\in[A]}GW_{A',1,3}(a,b,e_i)$ are finite, see \cite[Corollary $5.3.2$]{MS3}.
\end{rem}
The next theorem summarizes the main properties of these quantum products.

\begin{thm}\label{thm:properties of quantum products}(E.g. \cite[ Proposition $11.1.9$]{MS3}.)
\item[$(i)$] The quantum products $\ast_l, l=1,2,3$, are distributive over addition and super-commu-\\tative in the sense that
\begin{equation}\label{equation: super-commutativity of quantum products}
(a\otimes 1)\ast_l(b\otimes 1)=(-1)^{\deg(a)\deg(b)}(b\otimes 1)\ast_l(a\otimes 1)
\end{equation}
for $ l=1,2,3$ and for elements $a,b\in H_*(M;\FF)\cup H_*(M,\partial M;\FF)$ of pure degree.
\item[$(ii)$] The quantum products $\ast_l, l=1,2,3$, commute with the action of $\Lambda$, i.e. for any $\lambda\in\Lambda$ and for any $a,b\in QH_*(M)\cup QH_*(M,\partial M)$
\beq\label{equation: quantum product commutes with Lambda action}
\lambda(a\ast_lb)=(\lambda a)\ast_lb=a\ast_l (\lambda b)=(a\ast_l b)\lambda.
\eeq
\item[$(iii)$]  The quantum products $\ast_l, l=1,2,3$, are associative in the sense that
\bea\label{equation: associativity of quantum product }
&(a\ast_1 b)\ast_1c=a\ast_1 (b\ast_1c),\ \forall a, b, c\in QH_*(M),\\
&(a\ast_1 b)\ast_2c=a\ast_1 (b\ast_2c),\ \forall a, b\in QH_*(M),\ \forall  c\in QH_*(M,\partial M),\\
&(a\ast_2 b)\ast_2c=a\ast_2 (b\ast_3c),\ \forall a\in QH_*(M),\ \forall b, c\in QH_*(M,\partial M),\\
&(a\ast_3 b)\ast_3c=a\ast_3 (b\ast_3c),\ \forall a, b, c\in QH_*(M,\partial M).
\eea
\item[$(iv)$]The zero class term in the quantum products $\ast_l, l=1,2,3$, is the classical intersection products $\bullet_l, l=1,2,3$, i.e.
\beq\label{equation: zero class term is intersection product}
(a\ast_l b)_{A=0}=a\bullet_l b\otimes 1,\ \forall  a,b\in QH_*(M)\cup QH_*(M,\partial M).
\eeq
\item[$(v)$] The relative fundamental class $[M,\partial M]\otimes 1$ is the unit for $\ast_l, l=2,3$, i.e.
\bea\label{equation: fundamental class is the unit}
&a\ast_2 [M,\partial M]\otimes 1=a,\ \forall  a\in QH_*(M),\\
&a\ast_3 [M,\partial M]\otimes 1=a,\ \forall  a\in QH_*(M,\partial M).
\eea
\end{thm}

\begin{cor}\label{cor: algebraic structure of quantum homology}
\item[$(i)$]The triple $(QH_*(M), +, \ast_1)$ has the structure of a non-unital associative $\Lambda$-algebra.
\item[$(ii)$]The triple $(QH_*(M,\partial M), +, \ast_3)$ has the structure of a unital associative $\Lambda$-algebra, where $[M,\partial M]$ is the multiplicative unit.
\item[$(iii)$] The quantum product $\ast_2$ defines on $(QH_*(M), +, \ast_1)$ the structure of an associative algebra over the algebra $(QH_*(M,\partial M), +, \ast_3)$.
\end{cor}
Like in the closed case, we have different natural pairings. The $\KK_G$-valued pairings are given by
\bea\label{equation: Delta pairing on QH}
&\Delta_1:QH_k(M)\times QH_{2n-k}(M)\to\Lambda_0=\KK_G,\\
&\Delta_2:QH_k(M)\times QH_{2n-k}(M,\partial M)\to\Lambda_0=\KK_G,\\
&\Delta_l\(a, b\):=\imath(a\ast_l b),\ \text {for}\ l=1,2,
\eea
where the map $$\imath:QH_0(M)=\bigoplus_iH_i(M;\FF)\otimes_{\ff}\Lambda_{-i}\to\KK_G$$ sends $[pt]\otimes f_0(s)+\sum_{m=1}^{2n} a_m\otimes f_m(s)q^{-m}$ to $f_0(s)$.  The $\FF$-valued pairings are given by
\bea\label{equation: Pi pairing on QH}
&\Pi_1:QH_k(M)\times QH_{2n-k}(M)\to\FF,\\
&\Pi_2:QH_k(M)\times QH_{2n-k}(M,\partial M)\to\FF,\\
&\Pi_l=\jmath\circ\Delta_l,\ \text {for}\ l=1,2,
\eea
where the map $\jmath:\KK_G\to\FF$ sends $f(s)=\sum_{\alpha}z_{\alpha}s^{\alpha}\in \KK_G$ to $z_0$.

\begin{prop}\label{prop: properties of Delta-Pi pairings}
The pairings $\Delta_l$ and $\Pi_l$ for  $l=1,2$  satisfy
\bea\label{equation: Delta-Pi}
&\Delta_l(a,b)=\Delta_2(a\ast_l b,[M,\partial M]\otimes 1),\\
&\Pi_l(a,b)=\Pi_2(a\ast_l b,[M,\partial M]\otimes 1),
\eea
for any quantum homology classes $a\in QH_k(M), b\in QH_{2n-k}(M,\partial M)$. Moreover, the pairings $\Delta_2$ and $\Pi_2$ are non-degenerate. Since the quantum homology groups are finite-dimensional $\KK_G$-vector spaces in each degree, it follows that the paring $\Delta_2$ gives rise to Poincar\'{e}-Lefschetz duality over the field $\KK_G$.
\end{prop}
\begin{proof}
The formulas \eqref{equation: Delta-Pi} follow immediately from \eqref{equation: fundamental class is the unit}. Let us prove the non-degeneracy  statement. If
$$a=\sum_{ i\geq 0}a_i\otimes f_i(s)q^{k-i},\ a_i\in H_i(M;\FF), f_i(s)\in\KK_G,$$
and
$$b=\sum_{j\geq 0}b_j\otimes g_j(s)q^{2n-k-j},\ b_j\in H_j(M,\partial M;\FF), g_j(s)\in\KK_G,$$
then we have that
$$\Delta_2(a,b)=\sum_{i\geq 0}(a_i\bullet_2 b_{2n-i})f_i(s)g_{2n-i}(s).$$
Suppose that  $a\neq 0$.  Then $\exists\ i\geq 0$ such that $a_i\otimes f_i(s)\neq 0$.  Since the intersection product $\bullet_2$ is non-degenerate, the Kronecker-dual element
$a_i^{\vee}\in H_{2n-i}(M,\partial M;\FF)$ is non-zero and $a_i\bullet_2 a_i^{\vee}\neq 0$. Consider the element
$0\neq b:=a_i^{\vee}\otimes(f_i(s))^{-1} q^{-k+i}\in QH_{2n-k}(M,\partial M)$. We get that
$$\Delta_2(a,b)=(a_i\bullet_2 a_i^{\vee})\otimes 1\neq 0$$
and
$$\Pi_2(a,b)=\jmath((a_i\bullet_2 a_i^{\vee})\otimes 1)=a_i\bullet_2 a_i^{\vee}\neq 0.$$
We conclude that  the pairings $\Delta_2$ and $\Pi_2$ are non-degenerate.
\end{proof}
\begin{rem}
In general, the intersection product $\bullet_1$ is degenerate. If it  is non-degenerate, the pairings $\Delta_1$ and $\Pi_1$ are also non-degenerate. It happens, for example, when $\partial M$ is a homology sphere.
\end{rem}

\begin{rem}\label{remark: valuation on QH}
The Novikov ring $\Lambda$ and quantum homology groups $QH_*(M,\partial M)$, $QH_*(M)$  admit the following valuation. Define a valuation $$\nu:\KK_G\to G\cup\{-\infty\}$$ on the field $\KK_G$   by
\beq\label{equation: valuation nu on Lambda}
\begin{cases}
&\nu\(f(s)=\sum\limits_{\alpha\in G}z_{\alpha}s^{\alpha}\):=\max\{\alpha| z_{\alpha}\neq 0\},\; f(s)\not\equiv 0\\
&\nu(0)=-\infty.
\end{cases}
\eeq
Extend $\nu $ to $\Lambda$ by $$\nu(\lambda):=\max\{\alpha|p_{\alpha}\neq 0\},$$ where $\lambda$ is uniquely represented by $$\lambda=\sum\limits_{\alpha\in G}p_{\alpha}s^{\alpha},\;\; p_{\alpha}\in\FF[q,q^{-1}].$$
Note that for all $\lambda, \mu\in\Lambda$ we have
$$
\nu(\lambda+\mu)\leq\max(\nu(\lambda), \nu(\mu)),\;\nu(\lambda\mu)=\nu(\lambda)+\nu(\mu),\;\nu(\lambda^{-1})=
-\nu(\lambda).
$$
Now,  any non-zero $a\in QH_*(M,\partial M)\(QH_*(M)\)$ can be uniquely written as $a=\sum_ia_i\otimes\lambda_i$, where $a_i\in H_*(M,\partial M; \FF)\(H_*(M; \FF)\)$ and $\lambda_i\in\Lambda$. Define $$\nu(a):=\max_i\{\nu(\lambda_i)\}.$$
\end{rem}

\subsection{Examples}
In the following examples we consider the base field $\FF=\QQ$.
\begin{ex}\label{example: quantum homology of BCP}
Let  $(M,\omega):=(\til{\BB^4} , \omega_{\delta})$ be the symplectic blow-up at the origin of the standard
unit $4$-ball in $(\CC^2,\omega_0)$ of  a small size $0<\delta<<1$. Then $\{[pt], E\}$ is the integral homogeneous basis of $H_*(M;\QQ)$ and $\{E^{\vee}, [M,\partial M]\}$ is the integral homogeneous basis of $H_*(M, \partial M;\QQ)$. Since $\omega_{\delta}(E)=\delta>0$ and $c_1(E)=1$, it follows
that $\Gamma=H_2^S(M)=\ZZ\la E\ra$ and $G=\cfrac{\delta}{2}\cdot\ZZ$, see \eqref{equation: Gamma group of spherical classes} and \eqref{equation: G group of periods} for the notation.
From \eqref{equation: GW of BCP } we find the quantum products:
\beq\label{equation: quantum products for BCP}
\begin{tabular}{|c|l|} \hline
$\ast_1$
&\begin{tabular}{l}
\\
$[pt]\ast_1 H_*(M;\QQ)=0$    \\
\\
$E\ast_1E=-[pt]+ E\otimes s^{-\delta}q^{-2}$  \\
\end{tabular}\\
\hline
$\ast_2$
&\begin{tabular}{l}
\\
$[pt]\ast_2 E^{\vee}=0$      \\
\\
$E\ast_2E^{\vee}=[pt]- E\otimes s^{-\delta}q^{-2}$    \\
\end{tabular}\\
\hline
$\ast_3$
&\begin{tabular}{l}
\\
$E^{\vee}\ast_3E^{\vee}=-E^{\vee}\otimes s^{-\delta}q^{-2}$ \\
\\
$[M,\partial M]$ is the $\ast_3$-unit \\
\end{tabular}\\
\hline
\end{tabular}
\eeq
\newline
\end{ex}

\begin{ex}\label{example: quantum homology of DTSCP}
Let $(M,\omega):=(\til{X}, \omega_{\delta})$  be the symplectic blow-up of  the  cotangent unit disk bundle $X:=\DD T^*\Sigma$, see Example~\ref{example: GW of blowup of a  cotangent disk bundle of the surface}.  Then $\{[pt], \{a_i,b_i\}_{i=1}^g, [\Sigma], E\}$ is the integral homogeneous basis of $H_*(M;\QQ)$, and $\{[F], E^{\vee}, \{a^{\vee}_i,b^{\vee}_i\}_{i=1}^g, [M,\partial M]\}$ is the integral homogeneous basis of $H_*(M, \partial M;\QQ)$. Since $\omega_{\delta}(E)=\delta>0, \omega_{\delta}([\Sigma])=0 $ and $c_1(E)=1, c_1([\Sigma])=0$, it follows that $\Gamma=\ZZ\la E\ra$ and $G=\cfrac{\delta}{2}\cdot\ZZ$, see \eqref{equation: Gamma group of spherical classes} and \eqref{equation: G group of periods} for the notation. From \eqref{equation: GW of  DTSCP } we find the quantum products:
\begin{equation}\label{equation: quantum product1 for DTSCP}\hspace{-38mm}
\begin{tabular}{|c|l|} \hline
$\ast_1$
&\begin{tabular}{l}
\\
$[pt]\ast_1 H_*(M;\QQ)=0$\\
\\
$a_i\ast_1a_j=a_i\ast_1b_j=b_i\ast_1b_j=0$\\
\\
$[\Sigma]\ast_1a_i=[\Sigma]\ast_1b_i=[\Sigma]\ast_1E=0$\\
\\
$[\Sigma]\ast_1[\Sigma]=(2-2g)[pt]$\\
\\
$E\ast_1a_i=E\ast_1b_i=0$\\
\\
$E\ast_1E=-[pt]+ E\otimes s^{-\delta}q^{-2}$
\end{tabular}\\
\hline
\end{tabular}
\end{equation}

\begin{equation}\label{equation: quantum product2 for DTSCP}\hspace{-16mm}
\begin{tabular}{|c|l|} \hline
$\ast_2$
&\begin{tabular}{l}
\\
$z\ast_2w=z\bullet_2w$ for any basic $(z,w)\neq(E,E^{\vee})$\\
\\
$E\ast_2E^{\vee}=[pt]- E\otimes s^{-\delta}q^{-2}$\\
\\
$z\ast_2 [M,\partial M]=z, \forall z\in H_*(M;\QQ)$
\end{tabular}\\
\hline
\end{tabular}
\end{equation}

\begin{equation}\label{equation: quantum product3 for DTSCP}\hspace{-10mm}
\begin{tabular}{|c|l|} \hline
$\ast_3$
&\begin{tabular}{l}
\\
$[F]\ast_3[F]=[F]\ast_3E^{\vee}=[F]\ast_3a^{\vee}_i=[F]\ast_3b^{\vee}_i=0$\\
\\
$E^{\vee}\ast_3a^{\vee}_i=E^{\vee}\ast_3b^{\vee}_i=0$\\
\\
$E^{\vee}\ast_3E^{\vee}=-E^{\vee}\otimes s^{-\delta}q^{-2}$\\
\\
$a^{\vee}_i\ast_3a^{\vee}_j=a^{\vee}_i\ast_3b^{\vee}_j=b^{\vee}_i\ast_3b^{\vee}_j=0$\\
\\
$[M,\partial M]$ is the $\ast_3$-unit\
\end{tabular}\\
\hline
\end{tabular}
\end{equation}
\end{ex}

\end{document}